\documentclass[11pt]{article}

\usepackage{amssymb,amsmath}
\usepackage{amsthm}
\usepackage{hyperref}
\usepackage{mathrsfs}
\usepackage{textcomp}
\hypersetup{
    colorlinks,%
    citecolor=black,%
    filecolor=black,%
    linkcolor=black,%
    urlcolor=black
}

\usepackage{verbatim}

\usepackage{sectsty}

\makeatletter

\newdimen\bibspace
\renewenvironment{thebibliography}[1]{%
 \section*{\refname 
       \@mkboth{\MakeUppercase\refname}{\MakeUppercase\refname}}%
     \list{\@biblabel{\@arabic\c@enumiv}}%
          {\settowidth\labelwidth{\@biblabel{#1}}%
           \leftmargin\labelwidth
           \advance\leftmargin\labelsep
           \itemsep\bibspace
           \parsep\z@skip     %
           \@openbib@code
           \usecounter{enumiv}%
           \let\p@enumiv\@empty
           \renewcommand\theenumiv{\@arabic\c@enumiv}}%
     \sloppy\clubpenalty4000\widowpenalty4000%
     \sfcode`\.\@m}
    {\def\@noitemerr
      {\@latex@warning{Empty `thebibliography' environment}}%
     \endlist}

\makeatother

\makeatletter

\newtheorem{thm}{Theorem}[section]
\newtheorem{lem}[thm]{Lemma}
\newtheorem{prop}[thm]{Proposition}

\newtheorem{cor}[thm]{Corollary}
\newtheorem{rem}[thm]{Remark}

\numberwithin{equation}{section}

\def\XXint#1#2#3{{\setbox0=\hbox{$#1{#2#3}{\int}$}
  \vcenter{\hbox{$#2#3$}}\kern-.5\wd0}}

\newcommand{\al}{\alpha}                \newcommand{\lda}{\lambda}
                \newcommand{\pa}{\partial}
\newcommand{\va}{\varepsilon}           \newcommand{\ud}{\mathrm{d}}
\newcommand{\be}{\begin{equation}}      \newcommand{\ee}{\end{equation}}
                 
\newcommand{\Lda}{\Lambda}              
\newcommand{\R}{\mathbb{R}}              

\begin{document}

\title{\textbf{Schauder estimates for solutions of linear parabolic integro-differential equations}
\bigskip}

\author{\medskip Tianling Jin\quad and \quad Jingang Xiong }

\date{\today}

\maketitle

\begin{abstract}

We prove optimal pointwise Schauder estimates in the spatial variables for solutions of linear parabolic integro-differential equations. Optimal H\"older estimates in space-time for those spatial derivatives are also obtained.
\end{abstract}


\section{Introduction}

Integro-differential equations appear naturally when studying discontinuous stochastic process.  In a series of papers of Caffarelli-Silvestre \cite{CS09, CS10, CS11}, regularities of solutions of fully nonlinear integro-differential elliptic equations such as H\"older estimates, Cordes-Nirenberg type estimates and Evans-Krylov theorem were established. Regularity for parabolic integro-differential equations has been also studied, e.g., in \cite{CCV, LD1,LD2, FK, KS, JS} and many others. In this paper, we prove optimal pointwise Schauder estimates in the spatial variables for solutions of linear parabolic integro-differential equations. In general, we can not expect any interior continuity of the derivative of local solutions in the time variable even for the fractional heat equation $u_t+(-\Delta)^{\sigma/2} u=0$ without extra assumptions; see example 2.4.1 in \cite{LD2}.

We consider the linear parabolic integro-differential equation
\be \label{eq:linear}
u_t(x,t)-L u(x,t)= f(x,t) \quad \mbox{in }B_5\times(-5^\sigma,0],
\ee
where 
\be\label{eq:linear operator}
L u(x):=\int_{\R^n} \delta u(x,y;t)K(x,y;t)\,\ud y,
\ee
$\delta u(x,y;t)=u(x+y,t)+u(x-y,t)-2u(x,t)$ and $K(x,y;t)$ is a positive kernel.

We will restrict our attention to symmetric kernels which satisfy
\be\label{eq:symmetry}
K(x,y;t)=K(x,-y;t).
\ee
This assumption is somewhat implicit in the expression \eqref{eq:linear}. We also assume that the kernels are uniformly elliptic
\be\label{eq:elliptic}
\frac{(2-\sigma)\lda }{|y|^{n+\sigma}} \leq K(x,y;t)\leq \frac{(2-\sigma)\Lda }{|y|^{n+\sigma}}
\ee
for some $\sigma\in (0,2)$, $0<\lda\le\Lda<\infty$, which is an essential assumption leading to local regularizations. Finally,  we suppose that the kernels are $C^1$ away from the origin and satisfy
\be\label{eq:kernel smooth1}
|\nabla_y K(x,y;t)|\leq \frac{\Lda}{|y|^{n+\sigma+1}},
\ee
and in certain cases we assume more that the kernels are $C^2$ away  from the origin and satisfy
\be\label{eq:kernel smooth2}
|\nabla^2_y K(x,y;t)|\leq \frac{\Lda}{|y|^{n+\sigma+2}}.
\ee
These smoothness assumptions are usually used to reduce the influence of the boundary data in the exterior domain, and one of the consequences is that the solutions of translation invariant (or ``constant coefficients") equations will have high regularity. Moreover, the conditions \eqref{eq:kernel smooth1} and \eqref{eq:kernel smooth2} are scaling invariant, which will be used in our perturbative arguments.  We say that a kernel $K\in \mathscr{L}_0(\lda, \Lda, \sigma)$ if $K$ satisfies \eqref{eq:symmetry} and \eqref{eq:elliptic}, and $K\in \mathscr{L}_1(\lda, \Lda, \sigma)$ if $K$ satisfies \eqref{eq:symmetry}, \eqref{eq:elliptic} and \eqref{eq:kernel smooth1}. If in addition that $K$ satisfies \eqref{eq:kernel smooth2}, then we say that
 $K\in \mathscr{L}_2(\lda, \Lda, \sigma)$. 

In this paper, all the solutions of nonlocal equations are understood in the viscosity sense, where the definitions of such solutions and their many properties can be found in \cite{CS09} for elliptic equations and in \cite{LD1} for parabolic equations. One may also consider \emph{a priori} estimates for solutions of \eqref{eq:linear}, i.e., assuming a smooth function $u$ satisfies \eqref{eq:linear}. To obtain pointwise Schauder estimates for solutions of \eqref{eq:linear} at $x=0$, we assume that the kernel satisfies
\be\label{eq:holder K} 
\int_{\R^n}|K(x,y;t)-K(0,y;0)|\min(|y|^2,r^2)dy \le \Lambda (|x|^{\alpha}+|t|^{\frac{\al}{\sigma}}) r^{2-\sigma} 
\ee
for all $r\in (0, 1]$, $(x,t)\in B_5\times(-5^\sigma,0]$. \eqref{eq:holder K} means that $K$ is H\"older continuous at $(x,t)=(0,0)$ in some integral sense.
If $|K(x,y;t)-K(0,y;0)|\le \Lda(2-\sigma) (|x|^{\alpha}+|t|^{\frac{\al}{\sigma}}) |y|^{-n-\sigma}$, then one can check that \eqref{eq:holder K} holds. Meanwhile, we also assume that the right-hand side $f(x,t)$ is  H\"older continuous at $(x,t)=(0,0)$, i.e.,
\be\label{eq:holder f}
|f(x,t)-f(0,0)|\le M_f (|x|^\al+|t|^\frac{\al}{\sigma})\quad \mbox{and }|f(x,t)|\le M_f
\ee
for all $\ (x,t)\in B_5\times(-5^\sigma,0]$ with some nonnegative constant $M_f$.

For a real number $s$, $[s]$ denotes the largest integer which is less than or equals to $s$.  Our main result  is the following optimal pointwise Schauder estimate in spatial variables for solutions  of \eqref{eq:linear} with $K\in \mathscr{L}_2(\lda, \Lda, \sigma)$.

\begin{thm}\label{thm:linear schauder}
Let $K\in \mathscr{L}_2(\lda, \Lda, \sigma)$ with $2>\sigma\ge\sigma_0>0$. Let $\al\in (0,1)$ such that $|\sigma+\al-j|\ge\va_0$ for some $\va_0>0$, where $j=1,2,3$. Suppose that \eqref{eq:holder K} and \eqref{eq:holder f} hold. If $u$ is a viscosity solution of \eqref{eq:linear}, then there exists a polynomial $P(x)$ of degree $[\sigma+\al]$ such that for $x\in B_1$
\be\label{eq:optimal linear schauder}
\begin{split}
|u(x,0)-P(x)|&\leq C\left(\|u\|_{L^\infty(\R^n\times(-5^\sigma,0])}+ M_f\right)|x|^{\sigma+\al};\\
|\nabla^j P(0)|&\leq C\left(\|u\|_{L^\infty(\R^n\times(-5^\sigma,0])}+ M_f\right),\ j=0,\cdots, [\sigma+\al],
\end{split}
\ee
where $C$ is a positive constant depending only on $\lda,\Lda, n, \sigma_0, \al$ and $\va_0$.
\end{thm}

The constant $C$ in \eqref{eq:optimal linear schauder} does not depend on $\sigma$, and thus, does not blow up as $\sigma\to 2$. But it blows up as $\sigma+\al$ approaching to integers. The condition that $\sigma+\al$ is not an integer is necessary 
even for (elliptic) fractional Laplacian equation $(-\Delta)^{\sigma/2}u=f$; see, e.g., Chapter V in \cite{stein}.

 Various Schauder estimates for solutions of some linear elliptic nonlocal equations were obtained before in, e.g., \cite{BFV, BR, DK, KD} and global Schauder estimates for some  linear parabolic nonlocal equations with non-symmetric kernels were obtained in \cite{MP} using probabilistic arguments, compared to which a feature of our estimate \eqref{eq:optimal linear schauder} in Theorem \ref{thm:linear schauder} is that the solution $u$ of \eqref{eq:linear} is precisely of $C^{\sigma+\al}$ at $x=0$ provided $f$ is $C^{\al}$ at $x=0$. 
 
 In the case of second order parabolic equations, if the coefficients are of $C^\al_x$ in $x$ and only measurable in the time variable, then for a solution $u$ of such equations, its second order spatial derivatives $\nabla^2_x u$ are of $C^{\al,\al/2}_{x,t}$. Such results and related ones can be found in, e.g., \cite{BA,DSK,Knerr,KP,GL,LL,TW}. Similar optimal interior H\"older estimates in space-time for spatial derivatives of solutions of \eqref{eq:linear} will  follow from Theorem \ref{thm:linear schauder}; see Corollary \ref{cor:interior schauder} and Corollary \ref{cor:gradient holder} in Section \ref{sec:holder spatial derivatives}.
 In Theorem \ref{thm:linear schauder}, we require that $K$ and $f$ have regularity in $t$ at $t=0$ as well, which is needed in our compactness arguments for weak limits of nonlocal parabolic operators.

One common difficulty in approximation arguments to obtain regularities of solutions of nonlocal equations is to control the error of the tails at infinity, which results in a slight loss of regularity compared to second order equations, especially in the case when $\sigma+\al>1$ with $\sigma<1$, and in the case $\sigma+\al>2$. In this paper, we will approximate the genuine solution by solutions of ``constant coefficients" equations instead of polynomials, which is inspired by \cite{Ca,LN}. In this way, we do not need to take care of the tails at infinity, that leads to the optimal regularity. The only place where \eqref{eq:kernel smooth1} or \eqref{eq:kernel smooth2} is used is to obtain higher regularity of solutions of those corresponding ``constant coefficients" equations. 

In the following section, we prove the optimal pointwise Schauder estimates \eqref{eq:optimal linear schauder}. We first establish high regularity for solutions of translation invariant equations in Section \ref{sec:1}, which is the only place that we require $K$ is $C^2$ away from the origin especially for $\sigma+\al\in (2,3)$.   In Section \ref{sec:2} we use perturbative arguments to prove Theorem \ref{thm:linear schauder}. Section \ref{sec:holder spatial derivatives} is on the H\"older estimates in space-time for those spatial derivatives. In the Appendix, we recall some definitions and notions of nonlocal operators from \cite{CS10, LD2}, and establish two approximation lemmas for our own purposes, which are variants of those in \cite{CS10, LD2}.

\bigskip

\noindent\textbf{Acknowledgements:} We would like to thank Professor Luis Silvestre for many useful discussions and suggestions. We also thank Professor YanYan Li for his interests and constant encouragement. Tianling Jin was supported in part by NSF grant DMS-1362525. Jingang Xiong was supported in part by the First Class Postdoctoral Science Foundation of China (No. 2012M520002) and Beijing Municipal Commission of Education for the Supervisor of Excellent Doctoral Dissertation (20131002701).

\section{Optimal pointwise Schauder estimates in spatial variables}\label{sec:linear}

\subsection{Translation invariant equations}\label{sec:1}

In this section, we first establish good regularity on the solutions of  translation invariant equations, which is similar to ``constant coefficients" equation in the case of second order equations. 
\begin{prop} \label{lem:high regularity}
Suppose the kernel $K(y)\in \mathscr{L}_i(\lda, \Lda, \sigma)$ with $2-\sigma_2\ge\sigma\ge\sigma_1>1$ for some $\sigma_2>0$, where $i=1$ or $2$. If $v$ is a viscosity solution of
\[
v_t(x,t)-\int_{\R^n} \delta v(x,y;t)K(y)\,\ud y= g(x,t) \quad \mbox{in }B_8\times (-8^\sigma,0],
\]
where $g(\cdot, t)\in C_x^i(B_8)$ for all $t\in (-8^\sigma,0]$, then there exists a positive constant $ c_1$ depending only on $n,\lda, \Lda, \sigma_1$  such that
\be\label{eq:C2 est0}
\begin{split}
\sup_{t\in (-1,0)}\| v(\cdot, t)\|_{C^{1+i} (B_1)}\leq c_1(\|v\|_{L^\infty(\R^n\times (-8^\sigma,0])}+\sup_{t\in (-8^\sigma,0)}\|g(\cdot, t)\|_{C_x^{i}(B_8)});\\
\end{split}
\ee
and there exists another positive constant $ c_2$ depending only on $n,\lda, \Lda, \sigma_1,\sigma_2$  such that
\be\label{eq:C2 est00}
\begin{split}
\sup_{t\in (-1,0)}\| v(\cdot, t)\|_{C^{\sigma+i} (B_1)}\leq c_2(\|v\|_{L^\infty(\R^n\times (-8^\sigma,0])}+\sup_{t\in (-8^\sigma,0)}\|g(\cdot, t)\|_{C_x^{i}(B_8)}).\\
\end{split}
\ee
\end{prop}

This proposition will follow from the next lemma and standard integration by part techniques. 

\begin{lem} \label{lem:high regularity 0}
Let the kernel $K(y)\in \mathscr{L}_0(\lda, \Lda, \sigma)$ with $\sigma\ge\sigma_1>1$. Suppose that there exist two positive constants $ \tilde c$, and $ \bar\al\le 1$ satisfying $|\bar\al- \sigma+1|\ge\va_0$ for some $\va_0>0$, such that for every viscosity solution $u$ of
\[
u_t(x,t)-\int_{\R^n} \delta u(x,y;t)K(y)\,\ud y= 0 \quad \mbox{in }B_5\times (-5^\sigma,0],
\]
there holds
\be\label{eq:assumption c1 alpha}
\|\nabla_x u(\cdot, 0)\|_{C^{0,\bar\al} (B_1)}\leq \tilde c \|u\|_{L^\infty(\R^n\times (-5^\sigma,0])}.\\
\ee
Then there exists a positive constant $C$ depending only on $n,\lda, \Lda, \bar\al, \tilde c, \va_0$ and $\sigma_1$  such that for every viscosity solution $v$ of
\[
v_t(x,t)-\int_{\R^n} \delta v(x,y;t)K(y)\,\ud y= h(x,t) \quad \mbox{in }B_5\times (-5^\sigma,0],
\]
there holds
\be\label{eq:C1 alpha est0}
\begin{split}
\|\nabla_x v(\cdot, 0)\|_{C^{\beta} (B_1)}\leq C\left(\|v\|_{L^\infty(\R^n\times (-5^\sigma,0])}+\|h\|_{L^{\infty}(B_5\times(-5^\sigma,0])}\right),\\
\end{split}
\ee
where $\beta=\min(\sigma-1,\bar\al)$.
\end{lem}
Note that it follows from \cite{JS} that our assumption \eqref{eq:assumption c1 alpha} indeed holds for some $\bar\al>0$.  If we assume $K(y)\in \mathscr{L}_1(\lda, \Lda, \sigma)$ with $\sigma>1$, then by using Theorem 6.2 in \cite{LD1}, integration by part techniques and Lemma \ref{lem:high regularity 0} itself, we will see in the proof of Proposition \ref{lem:high regularity}  that \eqref{eq:assumption c1 alpha} actually holds with $\bar\al=1$, and thus, \eqref{eq:C1 alpha est0} holds with $\beta=\sigma-1$.
\begin{proof}
We can assume that $\|h(x,t)\|_{L^{\infty}(B_5\times(-5^\sigma,0])}+\|v\|_{L^\infty(\R^n\times (-5^\sigma,0])}\le 1$. Let $\rho=1/2$. For $k=0,1,2,\dots,$ let $Q_k=B_{\rho^k}\times(-\rho^{k\sigma},0]$ and $v_k$ be the solution of the following translation invariant equation
\[
\begin{split}
 \pa_t v_k(x,t)&-\int_{\R^n} \delta v_k(x,y;t)K(y)\,\ud y= 0 \quad \mbox{in }Q_k,\\
v_k&=v\quad\mbox{in }((\R^n\setminus B_{\rho^k})\times[-\rho^{k\sigma},0])\cup (\R^n\times\{t=-\rho^{k\sigma}\}).
\end{split}
\]
The existence and uniqueness of such $v_k$ is guaranteed by Theorem 3.3 in \cite{LD1}. Then we have by the maximum principle,
\[
 \|v_k-v\|_{L^\infty(\R^n\times [-\rho^{k\sigma},0])}\le \rho^{\sigma k}
\]
 and thus, by the maximum principle again,
\[
 \|v_k-v_{k+1}\|_{L^\infty(\R^n\times [-\rho^{(k+1)\sigma},0])}\le \|v_k-v\|_{L^\infty(\R^n\times [-\rho^{k\sigma},0])}\le \rho^{\sigma k}.
\]
Let $w_{k+1}=v_{k+1}-v_k$. It follows from the assumption estimate \eqref{eq:assumption c1 alpha} that for $x\in B_{\rho^{k+2}}$,
\[
 \begin{split}
   |\nabla_x w_{k+1}(x,0)|&\le C\rho^{(\sigma-1) k}\\
    | w_{k+1}(x,0)-w_{k+1}(0,0)-\nabla_x w_{k+1}(0,0)x|&\le C\rho^{(\sigma-1-\bar\al) k}|x|^{1+\bar\al}.
 \end{split}
\]
Thus, for $\rho^{i+2}\le |x|<\rho^{i+1}$, if we let $w=v-v_0$, then
\be\label{eq:smooth closeness}
\begin{split}
&|w(x,0)-\sum_{l=1}^\infty w_l(0,0)-\sum_{l=1}^\infty \nabla_x w_l(0,0)\cdot x |\\
&\le |w(x,0)-\sum_{l=1}^i w_l(x,0)|+|\sum_{l=1}^i w_l(x,0)-\sum_{l=1}^i w_l(0,0)-\sum_{l=1}^i \nabla_x w_l(0,0)\cdot x|\\
&\quad+|\sum_{l=i+1}^\infty w_l(0,0)|+|\sum_{l=i+1}^\infty \nabla_x w_l(0,0)\cdot x|\\
&\le \rho^{\sigma i}+C |x|^{1+\bar\al}\sum_{l=1}^i\rho^{(\sigma-1-\bar\al)l}+C\sum_{l=i+1}^\infty  \rho^{\sigma l}+C |x|\sum_{l=i+1}^\infty \rho^{(\sigma-1)l}\\
&\le C|x|^{\beta+1},
\end{split}
\ee
where $\beta=\min(\sigma-1, \bar\al)$, and $C$ depends only on $n,\lda, \Lda, \bar\al, \tilde c, \va_0$ and $\sigma_1$. Meanwhile, it follows from the assumption estimate \eqref{eq:assumption c1 alpha} that
\[
|v_0 (x,0)-v_0(0,0)-\nabla_x v_0(0,0) x|\le C|x|^{1+\bar\al}.
\]
This finishes the proof.
\end{proof}

\begin{proof}[Proof of Proposition \ref{lem:high regularity}]
First of all, we know from Theorem 6.2 in \cite{LD1} that $\nabla_x v$ is local H\"older continuous in space-time. We will use integration by parts techniques which can be found in \cite{CS09}. Let $\eta_1$ be a smooth cut-off function supported in $B_{7}$ and $\eta_1\equiv 1$ in $B_6$. Let $w_1=\nabla_x (\eta_1 v)$. Then it satisfies in viscosity sense that
\[
\begin{split}
\pa_t w_1 (x,t)-&\int_{\R^n} \delta w_1(x,y;t)K(y)\,\ud y\\
&= -\int_{\R^n}((1-\eta_1)v)(x+y,t)\nabla_yK(y)\ud y +\nabla_x g(x,t)\quad \mbox{in }B_5\times (-5^\sigma,0].
\end{split}
\]
Thus, if $K(y)\in \mathscr{L}_1(\lda, \Lda, \sigma)$, it follows from Lemma \ref{lem:high regularity 0} that $w_1$ is $C^{1+\beta}$ in $x$ for some $\beta>0$. Thus, we have $C^2$ estimate in $x$ \eqref{eq:C2 est0} for $v$. This implies that the assumption estimate \eqref{eq:assumption c1 alpha} in Lemma \ref{lem:high regularity 0} is satisfied with $\bar\al=1$ if $K(y)\in \mathscr{L}_1(\lda, \Lda, \sigma)$. Now, we apply Lemma \ref{lem:high regularity 0} once more to the equation of $w_1$. If we choose $\bar\al=1$ we have that $w_1$  is $C^{\sigma}$ in $x$, from which \eqref{eq:C2 est00} follows. This proves the case of $i=1$.

If $K(y)\in \mathscr{L}_2(\lda, \Lda, \sigma)$, then we take another smooth cut-off function $\eta_2$ supported in $B_{4}$ and $\eta_2\equiv 1$ in $B_3$. Let $w_2=\nabla_x (\eta_2w_1)$.  Then it satisfies
\[
\begin{split}
\pa_t w_2 (x,t)-&\int_{\R^n} \delta w_2(x,y;t)K(y)\,\ud y =  -\int_{\R^n}((1-\eta_2)w_1)(x+y)\nabla_yK(y)\ud y\\
&\quad+ \int_{\R^n}((1-\eta_1)v)(x+y)\nabla^2_yK(y)\ud y+\nabla^2_x g(x,t) \quad \mbox{in }B_{2}\times (-2^{\sigma},0].
\end{split}
\]
Thus, \eqref{eq:C2 est0} and \eqref{eq:C2 est00} follow as before.  This proves the case of $i=2$.
\end{proof}

Similarly, for $0<\sigma_0\le\sigma\le 1$, we have
\begin{lem} \label{lem:high regularity 01}
Let the kernel $K(y)\in \mathscr{L}_0(\lda, \Lda, \sigma)$ with $0<\sigma_0\le\sigma\le 1$. Suppose that there exist two positive constants $ \tilde c$, and $ \bar\al\le 1$ satisfying $|\bar\al-\sigma|\ge\va_0$ for some $\va_0>0$ such that for every viscosity solution $u$ of
\[
u_t(x,t)-\int_{\R^n} \delta u(x,y;t)K(y)\,\ud y= 0 \quad \mbox{in }B_5\times (-5^\sigma,0],
\]
there holds
\be\label{eq:assumption c alpha}
\| u(\cdot, 0)\|_{C^{0,\bar\al} (B_1)}\leq \tilde c \|u\|_{L^\infty(\R^n\times (-5^\sigma,0])}.\\
\ee
Then there exists a positive constant $C$ depending only on $n,\lda, \Lda, \bar\al, \tilde c, \va_0$ and $\sigma_0$  such that for every viscosity solution $v$ of
\[
v_t(x,t)-\int_{\R^n} \delta v(x,y;t)K(y)\,\ud y= h(x,t) \quad \mbox{in }B_5\times (-5^\sigma,0],
\]
there holds
\be\label{eq:C alpha est0}
\begin{split}
\| v(\cdot, 0)\|_{C^{\beta} (B_1)}\leq C\left(\|v\|_{L^\infty(\R^n\times (-5^\sigma,0])}+\|h\|_{L^{\infty}(B_5\times(-5^\sigma,0])}\right),\\
\end{split}
\ee
where $\beta=\min(\sigma,\bar\al)$.
\end{lem}
Note that it follows from Theorem 6.1 in \cite{LD1} that our assumption \eqref{eq:assumption c alpha} indeed holds for some $\bar\al>0$. If we assume $K(y)\in \mathscr{L}_1(\lda, \Lda, \sigma)$, then it follows from Theorem 6.2 in \cite{LD1} that \eqref{eq:assumption c alpha} actually holds with $\bar\al=1$.

\begin{prop} \label{lem:high regularity-1}
Suppose the kernel $K(y)\in \mathscr{L}_i(\lda, \Lda, \sigma)$ with $1\ge \sigma>\sigma_0>0$, where $i=1$ or $2$. If $v$ is a viscosity solution of
\[
v_t(x,t)-\int_{\R^n} \delta v(x,y;t)K(y)\,\ud y= g(x,t) \quad \mbox{in }B_8\times (-8^\sigma,0],
\]
where $g(\cdot, t)\in C_x^i(B_8)$ for all $t\in [-8^\sigma,0]$, then there exist a positive constant $ c_1$ depending only on $n,\lda, \Lda, \sigma_0$  such that
\[
\begin{split}
\sup_{t\in (-1,0)}\| v(\cdot, t)\|_{C^{i} (B_1)}\leq c_1(\|v\|_{L^\infty(\R^n\times (-8^\sigma,0])}+\sup_{t\in (-8^\sigma,0)}\|g(\cdot, t)\|_{C_x^{i}(B_8)}).\\
\end{split}
\]
When $\sigma\le 1-\va_0$ for some $\va_0>0$, there exist a positive constant $ c_2$ depending only on $n,\lda, \Lda, \sigma_0,\va_0$  such that
\[
\begin{split}
\sup_{t\in (-1,0)}\| v(\cdot, t)\|_{C^{\sigma+i} (B_1)}\leq c_2(\|v\|_{L^\infty(\R^n\times (-8^\sigma,0])}+\sup_{t\in (-8^\sigma,0)}\|g(\cdot, t)\|_{C_x^{i}(B_8)}).\\
\end{split}
\]
When $\sigma=1$, then for all $\beta\in (0,1)$ there exist a positive constant $ c_3$ depending only on $n,\lda, \Lda, \sigma_0,\beta$  such that
\[
\begin{split}
\sup_{t\in (-1,0)}\| v(\cdot, t)\|_{C^{\beta+i} (B_1)}\leq c_3(\|v\|_{L^\infty(\R^n\times (-8^\sigma,0])}+\sup_{t\in (-8^\sigma,0)}\|g(\cdot, t)\|_{C_x^{i}(B_8)}).\\
\end{split}
\]
\end{prop}

The proofs of Lemma \ref{lem:high regularity 01} and Proposition \ref{lem:high regularity-1} are very similar to those of Lemma \ref{lem:high regularity 0} and Proposition \ref{lem:high regularity}, respectively, and we leave them to the readers.

\subsection{Proof of the main theorem}\label{sec:2}

Now we are in position to prove Theorem \ref{thm:linear schauder} by approximations.
\begin{proof}[Proof of Theorem \ref{thm:linear schauder}]

The strategy of the proof is to find a sequence of approximation solutions which are sufficiently regular, and the error between the genuine solution and the approximation solutions can be controlled in a desired rate.

We may assume that $\|u\|_{L^\infty(\R^n\times(-5^\sigma,0])}+ M_f\le 1$. 
We claim that we can inductively find a sequence of functions $w_i$, $i=0,1,2,\cdots$, such that for all $i$,
\be\label{eq:for k}
\begin{split}
&\pa_t \sum_{l=0}^i w_l (x,t)-L_0 \sum_{l=0}^i w_l(x,t)=f(0,0)\quad \mbox{in }B_{4\cdot 5^{-i}}\times(-4^\sigma\cdot 5^{-i\sigma},0],
\end{split}
\ee
and
\be\label{eq:zero outside}
(u-\sum_{l=0}^i w_l)(5^{-i}x,5^{-i \sigma}t)=0\quad\mbox{in } ((\R^n\setminus B_4)\times[-4^\sigma,0])\cup (\R^n\times \{t=-4^\sigma\}),
\ee
and
\be\label{eq:close app}
\|u-\sum_{l=0}^i w_l\|_{L^\infty(\R^n\times[-4^\sigma5^{-i\sigma },0])}\le 5^{-(\sigma+\al)(i+1)},
\ee
and
\be\label{eq:est for k}
\begin{split}
\|w_i\|_{L^\infty(\R^n\times[-4^\sigma 5^{-i \sigma },0])}&\le 5^{-(\sigma+\al)i}\\
\|\nabla_x w_i\|_{L^\infty(B_{(4-\tau)\cdot 5^{-i}}\times [(-4^\sigma+\tau^\sigma)5^{-i\sigma},0])}&\le c_2 5^{-(\sigma+\al-1)i}\tau^{-1}\\
\|\nabla_x^2 w_i\|_{L^\infty(B_{(4-\tau)\cdot 5^{-i}}\times [(-4^\sigma+\tau^\sigma)5^{-i\sigma},0])}&\le c_2 5^{-(\sigma+\al-2)i}\tau^{-2},\\
\|\nabla_x^3 w_i\|_{L^\infty(B_{(4-\tau)\cdot 5^{-i}}\times [(-4^\sigma+\tau^\sigma)5^{-i\sigma},0])}&\le c_2 5^{-(\sigma+\al-3)i}\tau^{-3}\ (\text{if }\sigma>2-\al),\\
\end{split}
\ee
and
\be\label{eq:gradient close app}
[u-\sum_{l=0}^i w_l]_{C^{\al_1}( B_{(4-3\tau)\cdot 5^{-i}}\times[(-4^\sigma+2\tau^\sigma)5^{-i\sigma},0])}\le 8c_1 5^{i\al_1-(\sigma+\al)(i+1)}\tau^{-4},
\ee
where $\tau$ is an arbitrary constant in $(0,1)$,  $c_1>0$ and $\al_1\in (0,1)$ are positive constants depending only on $\lda,\Lda,n,\sigma_0$, and $c_2>0$ additionally depends on $\al$. Then, Theorem \ref{thm:linear schauder} follows from this claim and standard arguments. Indeed, as in \eqref{eq:smooth closeness}, we have, for $5^{-(i+1)}\le |x|<5^{-i}$,
\[
\begin{split}
|u(x,0)-\sum_{l=0}^\infty w_l(0,0)|
&\le C_1|x|^{\sigma+\al}\quad\mbox{when }\sigma+\al<1,\\
|u(x,0)-\sum_{l=0}^\infty w_l(0,0)-\sum_{l=0}^\infty \nabla_x w_l(0,0)\cdot x|&\le C_2|x|^{\sigma+\al}\quad\mbox{when }1<\sigma+\al<2.
\end{split}
\]
When $2<\sigma+\al<3$, we have, for $5^{-(i+1)}\le |x|<5^{-i}$, \[
\begin{split}
&|u(x,0)-\sum_{l=0}^\infty w_l(0,0)-\sum_{l=0}^\infty \nabla_x w_l(0,0)\cdot x-\sum_{l=0}^\infty \frac 12 x^T\nabla_x^2w_l(0,0)x|\\
&\le |u(x,0)-\sum_{l=0}^i w_l(x,0)|\\
&\quad+|\sum_{l=0}^i w_l(x,0)-\sum_{l=0}^i w_l(0,0)-\sum_{l=0}^i \nabla_x w_l(0,0)\cdot x-\sum_{l=0}^i \frac 12 x^T\nabla_x^2w_l(0,0)x|\\
&\quad+|\sum_{l=i+1}^\infty w_l(0,0)|+|\sum_{l=i+1}^\infty \nabla_x w_l(0,0)\cdot x|+\frac 12 |\sum_{l=i+1}^\infty x^T \nabla_x^2w_l(0,0)x|\\
&\le 5^{-(\sigma+\al)(i+1)}+2 c_2 |x|^{3}\sum_{l=0}^i5^{-(\sigma+\al-3)l}+\sum_{l=i+1}^\infty  5^{-(\sigma+\al)l}+|x|\sum_{l=i+1}^\infty c_2 5^{-(\sigma+\al-1)l}\\
&\quad+|x|^2\sum_{l=i+1}^\infty c_2 5^{-(\sigma+\al-2)l}\\
&\le C_3|x|^{\sigma+\al}.
\end{split}
\]
Note that we used $|\sigma+\alpha-j|\ge\va_0$ for j=1,2,3 in obtaining $C_1,C_2,C_3$, which actually blow up at a rate of $O(|\sigma+\al-j|^{-1})$ as $\sigma+\al\to j\in \{1,2,3\}$. The estimate \eqref{eq:optimal linear schauder} is proved using the claim.

Now we are left to prove this claim. Before we provide the detailed proof, we would like to first mention the idea and the structure of \eqref{eq:for k}-\eqref{eq:gradient close app}:
\begin{itemize}
\item Solving \eqref{eq:for k} and \eqref{eq:zero outside} inductively is how we construct this sequence of functions $\{w_i\}$.

\item \eqref{eq:close app} will follow from the approximation lemmas in the appendix, where \eqref{eq:gradient close app} will be used. 

\item \eqref{eq:est for k} will follow from \eqref{eq:close app}, maximum principles and the estimates in Proposition \ref{lem:high regularity} and Proposition \ref{lem:high regularity-1}. 

\end{itemize}

The proof of the above claim is by induction, and it consists of three steps.

\medskip

\noindent \emph{Step 1}: Normalization and rescaling.

\medskip
Let $w_0$ be the viscosity solution of
\be\label{eq:u0}
\begin{split}
\pa_t w_0-L_0 w_0&=f(0,0)\quad\mbox{in }B_4\times(-4^\sigma,0]\\
w_0&=u\quad\mbox{in }((\R^n\setminus B_4)\times[-4^\sigma,0])\cup (\R^n\times \{t=-4^\sigma\}),
\end{split}
\ee
where
\[
L_0 w=\int_{\R^n}\delta w(x,y;t)K(0,y;0)\,\ud y.
\]
We also think of $w_0\equiv u$ in $ \R^n\times (-5^\sigma,-4^\sigma)$. Then by comparison principles,
\be\label{eq:u0bound}
\|w_0\|_{L^\infty(\R^n\times [-4^\sigma,0])}\le  c_0(\|u\|_{L^\infty(\R^n\times (-5^\sigma,0])}+\|f\|_{L^\infty(B_5\times (-5^\sigma,0])}),
\ee
where $c_0$ is a positive constant depending only on $n,\lda,\Lda,\sigma_0$. By normalization, we may assume that
\[
\|w_0\|_{L^\infty(\R^n\times [-4^\sigma,0])}\le 1, \quad \|u\|_{L^\infty(\R^n\times [-5^\sigma,0])}+\|f\|_{L^\infty(B_5\times [-5^\sigma,0])}\le 1.
\]

For some universal small positive constant $\gamma<1$, which will be chosen in \eqref{eq:choice of eta}, we also may assume that $|f(x,t)-f(0,0)|\le\gamma (|x|^\al+|t|^{\frac{\al}{\sigma}})$ in $B_5\times (-5^\sigma,0]$  and 
\be\label{eq:holder K small} 
\int_{\R^n}|K(x,y;t)-K(0,y;0)|\min(|y|^2,r^2)dy \le  \gamma (|x|^{\alpha}+|t|^{\frac{\al}{\sigma}}) r^{2-\sigma} 
\ee
for all $r\in (0, 1]$, $(x,t)\in B_5\times(-5^\sigma,0]$. This can be achieved by the scaling
for $r<1$ small that if we let
\[
\begin{split}
\tilde K(x,y;t)&= r^{n+\sigma}K(rx,ry;r^\sigma t)\in \mathscr L_2(\lda,\Lda,\sigma),\\
 \tilde u(x,t)&=u(rx,r^\sigma t), \\
 \tilde f(x,t)&=r^{\sigma}f(rx, r^\sigma t),
\end{split}
\]
then we see that
\[
\tilde u_t (x,t)-\tilde L \tilde u(x,t)=\tilde f(x,t) \quad \mbox{in }B_5\times (-5^\sigma,0]
\]
where
\[
\tilde L \tilde u(x,t):=\int_{\R^n} \delta \tilde u(x,y;t) \tilde K(x,y,t)\,\ud y.
\]
Thus
\[
|\tilde f(x,t)-\tilde f(0,0)|\le M_f r^{\sigma+\al}(|x|^{\alpha}+|t|^{\frac{\al}{\sigma}})  \le \gamma (|x|^\al+|t|^{\frac{\al}{\sigma}})\le 10\gamma \quad \mbox{for small } r,
\]
in $B_{5}\times (-5^\sigma,0]$ and
\[
\begin{split}
\int_{\R^n}|\tilde K(x,y;t)-\tilde K(0,y;0)|\min(|y|^2,s^2)dy &\le2\Lda r^\al (|x|^{\alpha}+|t|^{\frac{\al}{\sigma}}) s^{2-\sigma} \\
&\le \gamma (|x|^{\alpha}+|t|^{\frac{\al}{\sigma}}) s^{2-\sigma} 
\end{split}
\]
for all $s\in (0, 1]$, $(x,t)\in B_5\times(-5^\sigma,0]$. It follows that ($\|\cdot\|_*$ is defined in \eqref{eq:norm star} in the Appendix)
\[
\|\tilde L-\tilde L_0\|_*\le50 \gamma\quad\mbox{ in } B_5\times(-5^\sigma,0].
\]
Indeed, for $(x,t)\in B_5\times(-5^\sigma,0]$, $\|h\|_{L^\infty(\R^n\times(-5^\sigma,0])}\le M$ and $|h(x+y,t)-h(x,t)-y\cdot \nabla_x h(x,t)|\le M |y|^2$ for every $y\in B_1$, we have
\be\label{eq:norm of the operator}
\begin{split}
&\|\tilde L-\tilde L_0\|_* \\
&\leq \sup_{(x,t), h}\frac{1 }{1+M} \int_{\R^n} |\delta h(x,y;t)||\tilde K (x,y;t)-\tilde K(0,y;0)|\,\ud y\\&
\le \sup_{(x,t)}\frac{M}{1+M} \Big( \int_{B_1} |y|^2|\tilde K (x,y;t)-\tilde K(0,y;0)|\ud y\\
&\quad +4\int_{\R^n\setminus B_1}|\tilde K (x,y;t)-\tilde K(0,y;0)|\ud y\Big) \\ 
&<50 \gamma.
\end{split}
\ee

\medskip
\noindent \emph{Step 2}: Prove the claim for $i=0$.

\medskip
Let $w_0$ be the one in Step 1. It follows from Proposition \ref{lem:high regularity} and Proposition \ref{lem:high regularity-1} that there exists a positive constant $c_2$ depending only on $\lda,\Lda,n,\sigma_0,\al$ such that
\be\label{eq:u0 C2}
\begin{split}
 \|w_0\|_{L^\infty(\R^n\times [-4^\sigma,0])}&\le 1\\
\|\nabla_x w_0\|_{L^\infty(B_{4-\tau}\times [-4^\sigma+\tau^\sigma,0])}&\leq c_2\tau^{-1}\\
\|\nabla^2_x w_0\|_{L^\infty(B_{4-\tau}\times [-4^\sigma+\tau^\sigma,0])}&\leq c_2\tau^{-2}\\
\|\nabla^3_x w_0\|_{L^\infty(B_{4-\tau}\times [-4^\sigma+\tau^\sigma,0])}&\leq c_2\tau^{-3}\ (\text{if }\sigma>2-\al).
\end{split}
\ee
For $\tau\in (0,1]$, it follows from Theorem 6.1 in \cite{LD1} (see \cite{CS09} for the elliptic case), standard scaling and covering (contributing at most a factor of $4/\tau$) argument that there exist constants $\al_1\in (0,1), c_1>0$, depending only on $n,\lda, \Lda, \sigma_0$,  such that
\be \label{eq:holder}
\|u\|_{C^{\al_1}(B_{4-\tau}\times [-4^\sigma+\tau^\sigma,0])} \leq c_1\tau^{-\al_1-1}.
\ee

Let us set up to apply the first approximation lemma in the Appendix, Lemma \ref{lem:appr0}.  Let $\va=5^{-(\sigma+\al)}$ and $M_1=1$ and let us fixed a modulus continuity $\rho(s)=s^{\al_1}$. Then for these $\rho,\va, M$, there exist $\eta_1$ (small) and $R$ (large) so that Lemma \ref{lem:appr0} holds. We can rescale the equation of $u$ so that it holds in a very large cylinder containing $B_{2R}\times[-(2R)^\sigma,0]$ and $|u(x,t)-u(y,s)|\le\rho(|x-y|\vee |t-s|)$ 
for every $(x,t)\in (B_R\setminus B_4)\times[-4^\sigma,0]$ and $(y,s)\in(\R^n\setminus B_4)\times[-4^\sigma,0]\cup \R^n\times\{s=-4^\sigma\}$. The latter one can be done due to \eqref{eq:holder}. And we will choose $\gamma<\eta_1/50$ in \eqref{eq:choice of eta}.
Then we can conclude from Lemma \ref{lem:appr0} that
\[
\|u-w_0\|_{L^\infty(B_4\times[-4^\sigma,0])}\leq \va= 5^{-(\sigma+\al)},
\]
and thus,
\[
\|u-w_0\|_{L^\infty(\R^n\times[-4^\sigma,0])}\le \|u-w_0\|_{L^\infty(B_4\times[-4^\sigma,0])}\leq \va= 5^{-(\sigma+\al)}.
\]
This proves \eqref{eq:for k}, \eqref{eq:zero outside}, \eqref{eq:close app} and \eqref{eq:est for k}  hold for $i=0$.

Moreover,
\be\label{eq:viscosity a bit}
|(u-w_0)_t-L (u-w_0)|\le 10 \gamma+(c_2+4)\gamma\tau^{-\sigma}
\ee
in $B_{4-2\tau}\times(-4^\sigma+\tau^\sigma,0]$ in viscosity sense. Indeed, let $t_0\in (0,1)$ and we smooth $w_0$ by using a mollifier $\eta_\va(x,t)$, and let $g_\va=\eta_\va* w_0$ (thinking of $w_0\equiv u$ in $\R^n\times[-5^\sigma,-4^\sigma]$). Let $w_0^\va$ be the solution of
\[
\begin{split}
\pa_t w_0^\va-L_0 w_0^\va&=f(0,0) \quad\mbox{in }B_4\times(-4^\sigma,-t_0]\\
w_0^\va&=g_\va\quad\mbox{in }((\R^n\setminus B_4)\times[-4^\sigma,-t_0])\cup (\R^n\times \{t=-4^\sigma\}).
\end{split}
\]
It follows from Theorem 4.1 in \cite{LD2} that $\pa_t v_{\va}$ is H\"older continuous in space-time. Thus, 
\[
(u-w_0^\va)_t-L (u-w_0^\va)=f(x,t)-f(0,0)+\int_{\R^n}\delta w_0^\va(x,y;t)(K(x,y;t)-K(0,y;0))\ud y.
\]
For $(x,t)\in B_{4-2\tau}\times[-4^\sigma+\tau^\sigma,t_0]$,  it follows from Proposition \ref{lem:high regularity} and Proposition \ref{lem:high regularity-1} that
\be\label{eq:cal error}
\begin{split}
&\int_{\R^n}|\delta w_0^\va(x,y;t)||(K(0,y;0)-K(x,y;t))|\ud y\\
&\le\int_{B_\tau}c_2\tau^{-2}|y|^2|(K(0,y;0)-K(x,y;t))|\ud y \\
&\quad+4\int_{\R^n\setminus B_\tau} |(K(0,y;0)-K(x,y;t))|\ud y\\
&\le (c_2+4)\gamma\tau^{-\sigma}(|x|^\al+|t|^\frac{\al}{\sigma})\le 10 (c_2+4)\gamma\tau^{-\sigma}.
\end{split}
\ee
Meanwhile, by the H\"older interior estimates, we have that $w_0^\va$ locally uniformly converges to some continuous function $w$. By the stability result Theorem 5.3 in \cite{LD2}, $w$ is a viscosity solution of 
\[
\begin{split}
\pa_t w-L_0 w &=f(0,0) \quad\mbox{in }B_4\times(-4^\sigma,-t_0]\\
w&=w_0\quad\mbox{in }((\R^n\setminus B_4)\times[-4^\sigma,-t_0])\cup (\R^n\times \{t=-4^\sigma\}).
\end{split}
\]
Hence $w\equiv w_0$. Thus, by sending $\va\to 0$, and $t_0\to 0$ with a standard perturbation argument ( using $\nu/t$ for small $\nu$), \eqref{eq:viscosity a bit} holds in $B_{4-2\tau}\times(-4^\sigma+\tau^\sigma,0]$ in viscosity sense.  By the choice of $\gamma$ in \eqref{eq:choice of eta},
\[
|(u-w_0)_t(x,t)-L (u-w_0)(x,t)|\le  10 \gamma + 10 (c_2+4)\gamma \tau^{-\sigma}\le 5^{-(\sigma+\al)}\tau^{-\sigma}.
\]
It follows from the H\"older estimates \eqref{eq:holder} proved in \cite{LD1}, standard  rescaling and covering arguments (contributing at most a factor of $4/\tau$) that
\[
[u-w_0]_{C^{\al_1}(B_{4-3\tau}\times[-4^\sigma+2\tau^\sigma,0])}\le \frac{4}{\tau}c_1\tau^{-\al_1}(5^{-(\sigma+\al)}\tau^{-\sigma}+5^{-(\sigma+\al)})\le 8c_15^{-(\sigma+\al)}\tau^{-4}.
\]
This finishes the proof of \eqref{eq:gradient close app}for $i=0$. 

\medskip

\noindent\emph{Step 3}: We assume all of \eqref{eq:for k}, \eqref{eq:zero outside}, \eqref{eq:close app}, \eqref{eq:est for k}, and \eqref{eq:gradient close app} hold up to $i\ge 0$. We will show that they all hold for $i+1$ as well. 
\medskip

Let
\[
W(x,t)=5^{(\sigma+\al)(i+1)}\left(u-\sum_{l=0}^i w_l\right)(5^{-(i+1)}x,5^{-(i+1)\sigma}t)
\]
and
\[
K^{(i+1)}(x,y;t)=5^{-(n+\sigma)(i+1)}K(5^{-(i+1)}x,5^{-(i+1)}y;5^{-(i+1)\sigma}t).
\]
Thus, by \eqref{eq:for k}, we have as before
\[
|W_t(x,t)-\int_{\R^n}\delta W(x,y;t)K^{(i+1)}(x,y;t)\ud y|\le A
\]
in viscosity sense in $B_{(4-2\tau)\cdot 5}\times[(-4^\sigma+\tau^\sigma)5^\sigma,0]$, where $A$ is a constant such that
\[
\begin{split}
& A\le |5^{\alpha (i+1)}(f(5^{-(i+1)}x,5^{-(i+1)\sigma}t)-f(0,0))|\\
&+\sum_{l=0}^i\int_{\R^n}5^{(\sigma+\al)(i+1)}|\delta w_l(5^{-(i+1)}x,5^{-(i+1)}y;5^{-(i+1)\sigma}t)|\cdot\\
&\quad\quad\quad\quad\quad\quad|K^{(i+1)}(x,y;t)-K^{(i+1)}(0,y;0)|\ud y.
\end{split}
\]
Then for $(x,t)\in B_{20}\times(-20^\sigma,0)$,
\[
\begin{split}
|5^{\alpha(i+1)}(f(5^{-(i+1)}x,5^{-(i+1)\sigma}t)-f(0,0))|\le 40\cdot \gamma,
\end{split}
\]
and for $l=0,1,\cdots, i$ and for $(x,t)\in B_{(4-2\tau)\cdot 5}\times[(-4^\sigma+\tau^\sigma)5^\sigma,0]$, we have, similar to \eqref{eq:cal error},
\be\label{eq:acc error}
\begin{split}
& 5^{(\sigma+\al)(i+1)}\int_{\R^n} |\delta w_l (5^{-(i+1)}x,5^{-(i+1)}y;5^{-(i+1)\sigma}t)|\cdot\\
&\quad\quad\quad\quad\quad\quad|K^{(i+1)}(x,y;t)-K^{(i+1)}(0,y;0)|\ud y|\\
&=5^{\al(i+1)} \int_{\R^n}|\delta w_l(5^{-(i+1)}x,y;5^{-(i+1)\sigma}t)||K(5^{-(i+1)}x,y;5^{-(i+1)\sigma}t)-K(0,y;0)|\ud y\\
&\le 5^{\al(i+1)}\int_{B_{5^{-l}\tau}} c_25^{-(\sigma+\al-2)l}\tau^{-2}|K(5^{-(i+1)}x,y;5^{-(i+1)\sigma}t)-K(0,y;0)|\ud y\\
&\quad + 5^{\al(i+1)}\int_{\R^n\setminus B_{5^{-l}\tau}}4\cdot 5^{-(\sigma+\al)l}|K(5^{-(i+1)}x,y;5^{-(i+1)\sigma}t)-K(0,y;0)|\ud y\\
&\le \gamma c_2\tau^{-\sigma}5^{-\al l}(|x|^\al+|t|^{\frac{\al}{\sigma}})+\gamma 4\tau^{-\sigma}5^{-\al l}(|x|^\al+|t|^{\frac{\al}{\sigma}})\\
&\le \gamma 40 (c_2+4)\tau^{-\sigma}5^{-\al l}.
\end{split}
\ee
 Thus, for  $(x,t)\in B_{(4-2\tau)\cdot 5}\times[(-4^\sigma+\tau^\sigma)5^\sigma,0]$, we have
\be\label{eq:for W}
\begin{split}
&|W_t(x,t)-\int_{\R^n} \delta W(x,y;t)K^{(i+1)}(x,y;t)\ud y|\\
&\le 40 \gamma+\gamma40 \left(c_2+4\right)\tau^{-\sigma}\sum_{l=0}^\infty5^{-\al l}\\
&= \tau^{-\sigma}\left(40+40(c_2+4)\sum_{l=0}^\infty5^{-\al l}\right)\gamma.\\
\end{split}
\ee
Let $\tau_0$ be such that $-4^\sigma+2\tau_0^\sigma<-2^\sigma$ which depends only on $\sigma_0$. Let $\eta_2<5^{-(\sigma+\al)}$ be as in Lemma \ref{lem:appr} with $M_2=1, M_3=8c_1, \beta=\al_1$ and $\va=5^{-(\sigma+\al)}$. We choose $\gamma$ such that
\be\label{eq:choice of eta}
\gamma<\eta_1/50\quad\mbox{and}\quad\tau_0^{-2}\left(40+40(c_2+4)\sum_{l=0}^\infty5^{-\al l}\right)\gamma\le\eta_2.
\ee
By our induction hypothesis \eqref{eq:close app}, \eqref{eq:gradient close app} and \eqref{eq:zero outside},
\[
\begin{split}
\|W\|_{L^{\infty}(\R^n\times [-20^\sigma,0])}&\le 1,\\
[ W]_{C^{\al_1}(B_{(4-3\tau)\cdot 5}\times [(-4^\sigma+2\tau^\sigma)5^\sigma,0]}&\le8c_1 5^{-\al_1}\tau^{-4}\le 8c_1\tau^{-4}.
\\
 W(x)=0 \quad\mbox{for all} &\quad (x,t)\in(\R^n\setminus B_{20})\times [-20^\tau,0]
 \end{split}
 \]
 It follows from \eqref{eq:for W}, \eqref{eq:choice of eta} and the choice of $\tau_0$ that
 \[
 |W_t(x,t)-\int_{\R^n} \delta W(x,y;t)K^{(i+1)}(x,y;t)\ud y|\le \eta_2\quad\mbox{for }(x,t)\in B_{10}\times (-10^\sigma,0].
 \]
Let $v_{i+1}$ be the solution of
\[
\begin{split}
\pa_t v_{i+1}&-\int_{\R^n}\delta v_{i+1}(x,y;t)K^{(i+1)}(0,y;0)\ud y=0\quad\mbox{in }B_4\times (-4^\sigma,0],\\
v_{i+1}&= W\quad\mbox{in }((\R^n\setminus B_{4})\times [-4^\tau,0])\cup (\R^n\times\{t=-4^\sigma\}).
\end{split}
\]
Together with the calculation in \eqref{eq:norm of the operator}, it follows from Lemma \ref{lem:appr} that
\be\label{eq:i+1 diff}
\|W-v_{i+1}\|_{L^{\infty}(\R^n\times [-4^\sigma,0])}=\|W-v_{i+1}\|_{L^{\infty}(B_4\times [-4^\sigma,0])}\le5^{-(\sigma+\al)}.
\ee
Moreover, it follows from \eqref{eq:close app} that
\[
\|v_{i+1}\|_{L^\infty(\R^n\times[-20^\sigma,0])}\le \|W\|_{L^\infty(\R^n\times[-20^\sigma,0])}\le 1,
\]
and thus by Proposition \ref{lem:high regularity} and Proposition \ref{lem:high regularity-1} (see also \eqref{eq:u0 C2})
\[
\begin{split}
&\|\nabla_x v_{i+1}\|_{L^\infty(B_{4-\tau}\times[-4^\sigma+\tau^\sigma,0])}\le c_2\tau^{-1},\\
& \|\nabla^2_x v_{i+1}\|_{L^\infty(B_{4-\tau}\times[-4^\sigma+\tau^\sigma,0])}\le c_2 \tau^{-2},\\
&\|\nabla^3_x v_{i+1}\|_{L^\infty(B_{4-\tau}\times[-4^\sigma+\tau^\sigma,0])}\le c_2\tau^{-3}\ (\text{if }\sigma>1).
\end{split}
\]
Define
\[
w_{i+1}(x,t)=5^{-(\sigma+\al)(i+1)}v_{i+1}(5^{i+1}x,5^{(i+1)\sigma}t).
\]
Then, \eqref{eq:for k}, \eqref{eq:zero outside},  \eqref{eq:close app} and \eqref{eq:est for k} hold for $i+1$. 

Moreover, for $(x,t)\in B_{4-2\tau}\times[-4^\sigma+\tau^\sigma,0]$, we have, similar to \eqref{eq:acc error},
\[
\begin{split}
&|\pa_t W(x,t)-\pa_t v_{i+1}(x,t)-\int_{\R^n}\delta (W(x,y;t)-v_{i+1}(x,y;t))K^{(i+1)}(x,y;t)\ud y|\\
&\le |5^{\al(i+1)}(f(5^{-(i+1)}x,5^{-(i+1)\sigma}t)-f(0,0))|\\
&+\sum_{l=0}^{i+1}\int_{\R^n}5^{(\sigma+\al)(i+1)}|\delta w_l(5^{-(i+1)}x,5^{-(i+1)}y;5^{-(i+1)\sigma}t)|\cdot\\
&\quad\quad\quad\quad\quad\quad|K^{i+1}(0,y;t)-K^{(i+1)}(x,y;t))|\ud y\\
&\le \tau^{-\sigma}\eta_2\le 5^{-(\sigma+\al)}\tau^{-\sigma},
\end{split}
\]
where we used \eqref{eq:choice of eta} in the second inequality. Thus, it follows from \eqref{eq:holder} and \eqref{eq:i+1 diff} that
\[
[W-v_{i+1}]_{C^{\al_1}(B_{4-3\tau}\times[-4^\sigma+2\tau^\sigma,0])}\le 8c_1 5^{-(\sigma+\al)}\tau^{-4}.
\]
Thus, \eqref{eq:gradient close app} hold for $i+1$ as well. This finishes the proof of the claim.
\end{proof}

A corollary of Theorem \ref{thm:linear schauder} would be the Schauder estimates for elliptic equations. If we consider the linear elliptic integro-differential equation
\be \label{eq:elliptic linear}
L u(x)= f(x) \quad \mbox{in }B_5
\ee
where 
\be\label{eq:elliptic linear operator}
L u(x):=\int_{\R^n} \delta u(x,y)K(x,y)\,\ud y,
\ee
$\delta u(x,y)=u(x+y)+u(x-y)-2u(x)$, $K(x,y)\in \mathscr L_2(\lda,\Lda, \sigma)$. We assume that 
\be\label{eq:elliptic holder K} 
\int_{\R^n}|K(x,y)-K(0,y)|\min(|y|^2,r^2)dy \le \Lambda |x|^{\alpha} r^{2-\sigma} 
\ee
for all $r\in (0, 1]$, $x\in B_5$, and 
\be\label{eq:elliptic holder f}
|f(x)-f(0)|\le M_f |x|^\al\quad \mbox{and }|f(x)|\le M_f\quad\forall \ x\in B_5
\ee
for some positive constant $M_f$. 
\begin{cor}\label{thm:elliptic linear schauder}
Let $K\in \mathscr{L}_2(\lda, \Lda, \sigma)$ with $2>\sigma\ge\sigma_0>0$. Let $\al\in (0,1)$ such that $|\sigma+\al-j|\ge\va_0$ for some $\va_0>0$, where $j=1,2,3$. Suppose that \eqref{eq:elliptic holder K} and \eqref{eq:elliptic holder f} hold. If $u$ is a viscosity solution of  \eqref{eq:elliptic linear}, then there exists a polynomial $P(x)$ of degree $[\sigma+\al]$ such that for $x\in B_1$
\be\label{eq:elliptic optimal linear schauder}
\begin{split}
|u(x)-P(x)|&\leq C\left(\|u\|_{L^\infty(\R^n)}+ M_f\right)|x|^{\sigma+\al};\\
|\nabla^j P(0)|&\leq C\left(\|u\|_{L^\infty(\R^n)}+ M_f\right),\ j=0,1, [\sigma+\al],
\end{split}
\ee
where $C$ is a positive constant depending only on $\lda,\Lda, n, \sigma_0, \al$ and $\va_0$.
\end{cor}

The constant $C$ in \eqref{eq:elliptic optimal linear schauder} does not blow up as $\sigma\to 2$, but it will blow up as $\sigma+\al$ approaches to integers.

\subsection{H\"older estimates in space-time for spatial derivatives}\label{sec:holder spatial derivatives}

Another corollary of the pointwise Schauder estimate \eqref{eq:optimal linear schauder} is the following uniform (in $t$) interior Schauder estimates in spatial variables.

We say that $K\in \mathscr{L}_2(\lda, \Lda, \sigma, \al)$ if $K\in \mathscr{L}_2(\lda, \Lda, \sigma)$ and 
\be\label{eq:K global holder}
\int_{\R^n}|K(x_1,y,t_1)-K(x_2,y,t_2)|\min(|y|^2,r^2)dy \le \Lambda (|x_1-x_2|^{\alpha} +|t_1-t_2|^{\frac{\al}{\sigma}})r^{2-\sigma} 
\ee
for all $r\in (0,1]$, $\ x_1,x_2\in B_5, t_1,t_2\in (-5^\sigma,0]$. We also assume that
\be\label{eq:f global holder}
|f(x_1,t_1)-f(x_2,t_2)|\le M_f(|x_1-x_2|^\al+|t_1-t_2|^{\frac{\sigma}{\al}}),\quad |f(x_1,t_1)|\le M_f
\ee
for all $\ x_1,x_2\in B_5, t_1,t_2\in (-5^\sigma,0]$ and some positive constant $M_f$.
\begin{cor}\label{cor:interior schauder}
Let $K\in \mathscr{L}_2(\lda, \Lda, \sigma,\al)$ with $2>\sigma\ge\sigma_0>0$. Let $\al\in (0,1)$ such that $|\sigma+\al-j|\ge\va_0$ for some $\va_0>0$, where $j=1,2,3$. Suppose that \eqref{eq:f global holder} holds. If $u$ is a viscosity solution of \eqref{eq:linear}, then $u(\cdot, t)\in C_x^{\sigma+\al}(B_1)$ for all $t\in (-1,0]$, and there exists a constant $C$ depending only on $\lda,\Lda, n, \sigma_0, \al, \va_0$, such that
\be\label{eq:optimal linear schauder1}
\sup_{t\in [-1,0]}\|u(\cdot,t)\|_{C^{\sigma+\al}(B_{1})}\leq C\left(\|u\|_{L^\infty(\R^n\times(-5^\sigma,0])}+ M_f\right).
\ee
\end{cor}
Corollary \ref{cor:interior schauder} follows from Theorem \ref{thm:linear schauder} and standard translation arguments. Once we know the optimal regularity estimates of $\nabla_x u$ or $\nabla_x^2 u$ in the spatial variables, we can also obtain their regularity estimates in the time variable.

We say that $K\in \mathscr{L}_3(\lda, \Lda, \sigma, \al)$ if $K\in \mathscr{L}_2(\lda, \Lda, \sigma, \al)$ and $|\nabla^3_y K(x,y;t)|\le \Lda |y|^{-n-\sigma-3}.$

\begin{cor}\label{cor:gradient holder}
Let $K\in \mathscr{L}_2(\lda, \Lda, \sigma,\al)$ with $2>\sigma\ge\sigma_0>0$. Let $\al\in (0,1)$ such that $|\sigma+\al-j|\ge\va_0$ for some $\va_0>0$, where $j=1,2,3$. Suppose that  \eqref{eq:f global holder} holds. If $u$ is a viscosity solution of \eqref{eq:linear}, then
\be\label{eq:lipschitz}
|u(x_1 ,t_1)-u(x_2 ,t_2)|\le C_1\left(\|u\|_{L^\infty(\R^n\times(-5^\sigma,0])}+ M_f\right)(|t_1-t_2|+|x_1-x_2|)
\ee
for all $x_1,x_2\in B_1$, $t_1,t_2\in (-1,0]$; if $1<\sigma+\al<2$, there holds
\be\label{eq:gradient holder in time}
\|\nabla_x u\|_{C_{x,t}^{\sigma+\al-1,\frac{\sigma+\al-1}{\sigma}}(B_1\times[-1,0])}\leq C_2\left(\|u\|_{L^\infty(\R^n\times(-5^\sigma,0])}+ M_f\right);
\ee
if $\sigma+\al>2$ and $K\in \mathscr{L}_3(\lda, \Lda, \sigma, \al)$, there hold \eqref{eq:gradient holder in time} and
\be\label{eq:hessian holder in time}
\|\nabla^2_x u\|_{C_{x,t}^{\sigma+\al-2,\frac{\sigma+\al-2}{\sigma}}(B_1\times[-1,0])}\leq C_3\left(\|u\|_{L^\infty(\R^n\times(-5^\sigma,0])}+ M_f\right),
\ee
where $ C_1, C_2, C_3$ are positive constants depending only on $\lda,\Lda, n, \sigma_0, \al, \va_0$.
\end{cor}
In particular, the constants $C, C_1,C_2,C_3$ in \eqref{eq:optimal linear schauder1}-\eqref{eq:hessian holder in time} do not blow up as $\sigma\to 2^-$.

\begin{lem}\label{lem:lip in time}
Let $v\in C(\R^n\times [-5^\sigma,0])$ satisfies \eqref{eq:u0} in viscosity sense. Then $v$ is locally Lipschitz in time. Moreover, for $t\in [-1,0]$ there holds
\[
|v(0,t)-v(0,0)|\le C\left(\|v\|_{L^\infty(\R^n\times [-4^\sigma,0])}+|f(0,0)|\right)|t|.
\]
\end{lem}
By using \eqref{eq:u0 C2} and the equation \eqref{eq:u0} itself, this estimate is clear if we consider it as a priori estimate.
\begin{proof}[Proof of Lemma \ref{lem:lip in time}]
Let $t_0\in (0,4^\sigma)$. To proceed, we smooth $v$ by using a mollifier $\eta_\va(x,t)$, and let $g_\va=\eta_\va* v$. Let $v_\va$ be the solution of
\[
\begin{split}
\pa_t v_\va-L_0 v_\va&=f(0,0) \quad\mbox{in }B_4\times(-4^\sigma,-t_0]\\
v_\va&=g_\va\quad\mbox{in }((\R^n\setminus B_4)\times[-4^\sigma,-t_0])\cup (\R^n\times \{t=-4^\sigma\}).
\end{split}
\]
It follows from Theorem 4.1 in \cite{LD2} that $\pa_t v_{\va}$ is H\"older continuous in space-time. By Proposition \ref{lem:high regularity} and Proposition \ref{lem:high regularity-1}, we know that $v_\va$ is $C^2$ in $x$. Thus, $v_\va$ satisfies its equation in the classical sense. By the equation of $v_\va$, 
\[
\begin{split}
&\|\pa_t v_\va\|_{L^\infty(B_1\times[-1,-t_0])}\\
&\quad\le \|L_0 v_\va\|_{L^\infty(B_1\times[-1,-t_0])}+|f(0,0)|\\
&\quad\le C (\|\nabla^2_xv_\va\|_{L^\infty(B_3\times[-1,-t_0])}+\|v_\va\|_{L^\infty(\R^n\times[-1,-t_0])}+|f(0,0)|)\\
&\quad\le C\left(\|v_\va\|_{L^\infty(\R^n\times [-4^\sigma,-t_0])}+|f(0,0)|\right)\\
&\quad\le C\left(\|v\|_{L^\infty(\R^n\times [-4^\sigma,-t_0])}+|f(0,0)|\right),
\end{split}
\]
where the estimates in Proposition \ref{lem:high regularity} and Proposition \ref{lem:high regularity-1} are used in the third inequality. Meanwhile, by the H\"older interior estimates, we have that $v_\va$ locally uniformly converges to some continuous function $w$. By the stability result Theorem 5.3 in \cite{LD2}, $w$ is a viscosity solution of 
\[
\begin{split}
\pa_t w-L_0 w &=f(0,0) \quad\mbox{in }B_4\times(-4^\sigma,-t_0]\\
w&=v\quad\mbox{in }((\R^n\setminus B_4)\times[-4^\sigma,-t_0])\cup (\R^n\times \{t=-4^\sigma\}).
\end{split}
\]
Hence $w\equiv v$, and thus, we have that
\[
\begin{split}
|v(0,t)-v(0,-t_0)|&=\lim_{\va\to 0}|v_\va(0,t)-v_\va(0,-t_0)|\\
&\le C\left(\|v\|_{L^\infty(\R^n\times [-4^\sigma,0])}+|f(0,0)|\right)|t+t_0|.
\end{split}
\]
We finish the proof by sending $t_0\to 0$. 
\end{proof}
\begin{rem}\label{rem:1}
Indeed, by similar arguments and the integration by parts technique used in the proof of Proposition \ref{lem:high regularity} one can also show that $\nabla_x v$ is Lipschitz in time, as well as $\nabla^2_x v$ if $K\in \mathscr{L}_3(\lda, \Lda, \sigma, \al)$ and $\sigma>1$ (so that we have estimates for $\nabla_x^4 v$).  We omit the proof here.
\end{rem}

\begin{proof}[Proof of Corollary \ref{cor:gradient holder}] If we let $w_l$ be as in the proof of Theorem \ref{thm:linear schauder}, then by Lemma \ref{lem:lip in time} and Remark \ref{rem:1}, we have that $w_l, \nabla_x w_l$ is Lipschitz in time, as well as $\nabla^2_x w_l$ provided that $ K\in \mathscr{L}_3(\lda, \Lda, \sigma, \al)$. By Corollary \ref{cor:interior schauder}, we may assume that $x_1=x_2=0, t_2=0.$ Suppose that $-5^{i\sigma}\le t<-5^{(i+1)\sigma}$. Then we have
\[
\begin{split}
&| u(0,t)- u(0,0)|\\
&\le| u(0,t)-\sum_{l=0}^i  w_l(0,t)|+| u(0,0)-\sum_{l=0}^i w_l(0,0)|+ \sum_{l=0}^i| w_l(0,t)-  w_l(0,0)|\\
&=2\cdot 5^{-(\sigma+\al)(i+1)}+C\sum_{l=0}^i 5^{-\al l}|t|\le C_1|t|,
\end{split}
\]
which proves \eqref{eq:lipschitz}. 

Suppose that $1<\sigma+\al<2$. We have
\[
\begin{split}
|\nabla_x u(0,t)-\nabla_x u(0,0)|&\le|\nabla_x u(0,t)-\sum_{l=0}^i \nabla_x w_l(0,t)|\\
&\quad+\sum_{l=0}^i|\nabla_x w_l(0,t)- \nabla_x w_l(0,0)|\\
&\quad+|\nabla_x u(0,0)-\sum_{l=0}^i \nabla_x w_l(0,0)|:=I_1+I_2+I_3.
\end{split}
\]
Since $\nabla_x u(0,0)=\sum_{l=0}^\infty \nabla_x w_l(0,0)$, we have
\[
 |I_3|\leq \sum_{l=i+1}^\infty c_2 5^{-(\sigma+\al-1)l}\le c_2|t|^{\frac{\sigma+\al-1}{\sigma}}\frac{1}{1-5^{-(\sigma+\al-1)}}.
\]
By the equation of $w_l$, Lemma \ref{lem:lip in time}, Remark \ref{rem:1} and \eqref{eq:est for k}, we have
\[
 |\nabla_x w_l(0,t)-\nabla_x w_l(0,0)|\le C 5^{(1-\al)l}|t|.
\]
Then
\[
\begin{split}
 |I_2|\le  C|t|\frac{5^{(1-\al)(i+1)}}{5^{1-\al}-1}\le C|t|^\frac{\sigma+\al-1}{\sigma} \frac{5^{1-\al}}{5^{1-\al}-1}.
\end{split}
\]
Meanwhile, it follows from the estimate \eqref{eq:optimal linear schauder} that
\[
 |u(x,t)-u(0,t)-\nabla_xu(0,t)x|\le C|x|^{\sigma+\al}.
\]
For $5^{-(i+1)}\le |x|<5^{-i}$, we have, by triangle inequality, 
\[
 \begin{split}
 &|\nabla_x u (0,t)x-\sum_{l=0}^i\nabla_x w_l(0,t)x|\\
 &\le |u(x,t)-\sum_{l=0}^i w_l(x,t)|+|\sum_{l=0}^i (w_l(x,t)-w_l(0,t)-\nabla_x w_l(0,t)x)|\\
&\quad+|u(0,t)+\nabla_x u(0,t)x-u(x,t)|+|\sum_{l=0}^i w_l(0,t)-u(0,t)|\\
&\le 5^{-(\sigma+\al)(i+1)}+c_2\sum_{l=0}^i 5^{-(\sigma+\al-2)l}|x|^2+ C|x|^{\sigma+\al}+5^{-(\sigma+\al)(i+1)}\\
 \end{split}
\]
Thus
\[
 \begin{split}
 |I_1|=|\nabla_x u (0,t)-\sum_{l=0}^i\nabla_x w_l(0,t)|&\le (C+\frac{1}{5^{2-\sigma-\al}-1})5^{-(\sigma+\al-1)(i+1)}\\
&\le (C+\frac{1}{5^{2-\sigma-\al}-1})|t|^{\frac{\sigma+\al-1}{\sigma}}.
 \end{split}
\]
Hence, we have shown that
\[
 |\nabla_x u(0,t)-\nabla_x u(0,0)|\le C_2|t|^\frac{\sigma+\al-1}{\sigma}.
\]

Suppose that $2<\sigma+\al<3$. We have
\[
\begin{split}
|\nabla^2_x u(0,t)-\nabla^2_x u(0,0)|&\le|\nabla^2_x u(0,t)-\sum_{l=0}^i \nabla^2_x w_l(0,t)|\\
&\quad+\sum_{l=0}^i|\nabla^2_x w(0,t)- \nabla^2_x w_l(0,0)|\\
&\quad+|\nabla^2_x u(0,0)-\sum_{l=0}^i \nabla^2_x w_l(0,0)|:=II_1+II_2+II_3.
\end{split}
\]
Since $\nabla^2_x u(0,0)=\sum_{l=0}^\infty \nabla^2_x w_l(0,0)$, we have
\[
 |II_3|\leq c_2|t|^{\frac{\sigma+\al-2}{\sigma}}\frac{1}{1-5^{-(\sigma+\al-2)}}.
\]
Meanwhile, it follows from the estimate \eqref{eq:optimal linear schauder} that
\[
 | u(x,t)- u(0,t)-\nabla_x u(0,t)x-\frac 12 x^T \nabla^2_x u(0,t)x|\le C|x|^{\sigma+\al}.
\]
By triangle inequality and the estimate for $I_1$, we have, for $5^{-(i+1)}\le |x|<5^{-i}$
\[
 \begin{split}
 &\frac 12 |x^T\nabla^2_x u (0,t)x-x^T\sum_{l=0}^i\nabla^2_x w_l(0,t)x|\\
&\le |u(x,t)-\sum_{l=0}^i w_l(x,t)|+|u(0,t)+\nabla_x u(0,t)x+\frac 12 x^T \nabla^2_x u(0,t)x-u(x,t)|\\
&\quad+|\sum_{l=0}^i (w_l(x,t)-w_l(0,t)-\nabla_x w_l(0,t)x-\frac 12 x^T \nabla^2_x w_l(0,t)x)|\\
&\quad+|\sum_{l=0}^i w_l(0,t)-u(0,t)|+|\sum_{l=0}^i\nabla_x w_l(0,t)x-\nabla_x u (0,t)x|\\
&\le 5^{-(\sigma+\al)(i+1)}+c_2\sum_{l=0}^i 5^{-(\sigma+\al-3)l}|x|^3+ C|x|^{\sigma+\al}+5^{-(\sigma+\al)(i+1)}\\
&\le 2\cdot 5^{-(\sigma+\al)(i+1)}+c_2\frac{5^{-(\sigma+\al)(i+1)}}{5^{2-\sigma-\al}-1}+ C5^{-(\sigma+\al)i}+(C+\frac{1}{5^{2-\sigma-\al}-1})5^{-(\sigma+\al)(i+1)}.
 \end{split}
\]
Thus
\[
 \begin{split}
 |II_1|=|\nabla_x u (0,t)-\sum_{l=0}^i\nabla_x w_l(0,t)|&\le C5^{-(\sigma+\al-2)(i+1)}\le C|t|^{\frac{\sigma+\al-2}{\sigma}}.
 \end{split}
\]
By the equation of $w_l$, Lemma \ref{lem:lip in time}, Remark \ref{rem:1}   and \eqref{eq:est for k}, we have
\[
 |\nabla^2_x w_l(0,t)-\nabla^2_x w_l(0,0)|\le C_3 5^{(2-\al)l}|t|
\]
provided that $ K\in \mathscr{L}_3(\lda, \Lda, \sigma, \al)$.
Then
\[
\begin{split}
 |I_2|\le  C|t|\frac{5^{(2-\al)(i+1)}}{5^{2-\al}-1}\le C|t|^\frac{\sigma+\al-2}{\sigma} \frac{5^{2-\al}}{5^{2-\al}-1}.
\end{split}
\]
Thus, by combining the estimates for $II_1, II_2, II_3$, we have that
\[
 |\nabla^2_x u(0,t)-\nabla^2_x u(0,0)|\le C|t|^\frac{\sigma+\al-2}{\sigma}.
\]
This completes the proof of Corollary \ref{cor:gradient holder}.
\end{proof}
If we do not assume $ K\in \mathscr{L}_3(\lda, \Lda, \sigma, \al)$ when $\sigma+\al>2$, we have that $\nabla^2_x u$ is of $C^\beta$ in the time variable for some $\beta>0$. This is because $\nabla_x^2 u$ is H\"older continuous in $x$ and $\nabla_x u$ is H\"older continuous in $t$, which implies that $\nabla_x^2 u$ is H\"older continuous in $t$ as well; see Lemma 3.1 on page 78 in \cite{LSU}.

\appendix

\section{Approximation lemmas}
Our proof of Schauder estimates uses perturbative arguments, and we need the following two approximation lemmas, which are variants of Theorem 5.6 in \cite{LD2} (Lemma 7 in \cite{CS10} in elliptic cases). We will do a few modifications for our own purposes, and we decide to include them in this appendix for completeness and convenience.  If it is just for our particular linear equations, those approximation lemmas can be simplified much. But we would like to include nonlinear equations as well in this step.  

To start with, we recall some definitions and notations about nonlocal operators, which can be found in \cite{LD1,LD2} for parabolic cases and in \cite{CS09,CS10} for elliptic cases.  Let $\sigma_0\in (0,2)$ be fixed, and $\omega(y)=(1+|y|^{n+\sigma_0})^{-1}$. We say $u\in L^1(\R^n,\omega)$ if $\int_{\R^n}|u(y)|\omega(y)\ud y<\infty.$  We say that $u\in C(a,b;L^1(\omega))$ if $u(\cdot, t)\in L^1(\R^n,\omega)$ for every $t\in (a,b)$, and $\|u(\cdot, t_1)-u(\cdot,t_2)\|_{L^1(\R^n,\omega)}\to 0$ as $t_1\to t_2^-$ for every $t_2\in (a,b)$, and we denote $\|u\|_{C(a,b;L^1(\omega))}=\sup_{t\in (a,b)}\|u(\cdot, t)\|_{L^1(\R^n,\omega)}$.

Nonlocal (continuous) operators $I$ are defined as ``black boxes" in Definitions 3.3 and 3.6 in \cite{LD1} such that, rough speaking, if $u$ is a test function at $(x,t)$, the $Iu$ is continuous near $(x,t)$. In our case, they are just linear operators of the form \eqref{eq:linear operator} with some continuity assumptions on $K$ in $(x,t)$. Sometimes, we also write $Iu(x,t)$ as $I(u,x,t)$ for convenience especially when dealing with $I(u+v)$.

An operator is translation invariant  if $\tau_{(z,s)} Iu=I(\tau_{(z,s)} u)$ where $\tau_{(z,s)}$ is the translation operator $\tau_{(z,s)}u(x,t)=u(x-z,t-s)$.

Given such a nonlocal operator $I$ defined on $\Omega\times (-T,0]$,  a norm $\|I\|$ was defined in Definition 5.3 in \cite{LD2}.
Here, we also define a (weaker) norm $\|I\|_{*}$ for our own purpose, 
\be\label{eq:norm star}
 \begin{split}
  \|I(t)\|_*:=& \sup\{|Iu(x,t)|/(1+M):x\in\Omega,  \|u(\cdot,t)\|_{\R^n}\le M\\
          &|u(x+y,t)-u(x,t)-y\cdot \nabla_x u(x,t)|\le M|y|^2 \mbox{ for every }y\in B_1\},
 \end{split}
\ee
and $
\|I\|_*=\sup_{t\in (-T,0]}\|I(t)\|_*.$ 

We say that a nonlocal operator $I$ is uniformly elliptic with respect to $\mathscr{L}_0(\lda,\Lda,\sigma)$, which will be written as $\mathscr{L}_0(\sigma)$ for short, if
\be\label{eq:elliptic operator}
\mathcal M^-_{\mathscr L_0(\sigma)} v(x,t)\le I(u+v,x,t)-I(u,x,t)\le \mathcal  M^+_{\mathscr L_0(\sigma)} v(x,t),
\ee
where
\[
 \mathcal M^-_{\mathscr L_0(\sigma)} v(x,t)=\inf_{L\in L_0(\sigma)}Lv(x,t)=(2-\sigma)\int_{\R^n}\frac{\lda\delta v(x,y;t)^+-\Lda\delta v(x,y;t)^-}{|y|^{n+\sigma}}\ud y
 \]
 \[
\mathcal M^+_{\mathscr L_0(\sigma)} v(x,t)=\sup_{L\in L_0(\sigma)}Lv(x,t)=(2-\sigma)\int_{\R^n}\frac{\Lda\delta v(x,y;t)^+-\lda\delta v(x,y;t)^-}{|y|^{n+\sigma}}\ud y.
\]
It is also convenient to define the limit operators when $\sigma\to 2$ as
\[
 \begin{split}
   \mathcal M^-_{\mathscr L_0(2)} v(x,t)=\lim_{\sigma\to 2}\mathcal  M^-_{\mathscr L_0(\sigma)} v(x,t)\\
\mathcal M^+_{\mathscr L_0(2)} v(x,t)=\lim_{\sigma\to 2}\mathcal  M^+_{\mathscr L_0(\sigma)} v(x,t).
 \end{split}
\]
It has been explained in \cite{CS10} that $\mathcal M^+_{\mathscr L_0(2)}$ is a second order uniformly elliptic operator, whose ellipticity constants $\tilde \lda$ and $\tilde \Lda$ depend only $\lda,\Lda$ and the dimension $n$. Moreover, $\mathcal M^+_{\mathscr L_0(2)} v\le \mathcal M^+(\nabla^2 v)$, where $\mathcal M^+(\nabla^2 v)$ is the second order Pucci operator with ellipticity constants $\tilde \lda$ and $\tilde \Lda$. Similarly, we also have corresponding relations for $ \mathcal M^-_{\mathscr L_0(2)}$.

For compactness arguments, we shall use the concept of the weak convergence of nonlocal operators, which can be found in Definition 5.1 in \cite{LD2} (Definition 41 in \cite{CS10} in the elliptic cases).

\begin{lem}\label{lem:appr0}
For some $\sigma\ge \sigma_0>0$ we consider nonlocal continuous operators $I_0$, $I_1$ and $I_2$ uniformly elliptic with respect to $\mathscr{L}_0(\sigma)$. Assume also that $I_0$ is translation invariant and $I_0 0=1$.

 Given $M_1>0$, a modulus of continuity $\rho$ and $\va>0$, there exist $\eta_1$ (small, independent of $\sigma$) and $R$ (large, independent of $\sigma$) so that if $u,v,I_0,I_1$ and $I_2$ satisfy
\[
\begin{split}
v_t-I_0 (v,x,t)&=0 \quad \mbox{in }B_1\times (-1,0],\\
u_t-I_1 (u,x,t)&\le \eta_1 \quad \mbox{in }B_1\times (-1,0],\\
u_t-I_2 (u,x,t)&\ge -\eta_1 \quad \mbox{in }B_1\times (-1,0],\\
\end{split}
\]
in viscosity sense, and
\[
\begin{split}
\|I_1-I_0\|_*&\le\eta_1\quad \mbox{in }B_1\times (-1,0],\\
\|I_2-I_0\|_*&\le\eta_1\quad \mbox{in }B_1\times (-1,0],\\
u&=v  \quad \mbox{in }((\R^n\setminus B_1)\times[-1,0])\cup(B_1\times\{t=-1\}),\\
|u|&\le M_1 \quad \mbox{in }\R^n\times[-1,0],\\
\end{split}
\]
and for every $(x,t)\in ((B_R\setminus B_1)\times (-1,0])\cup (B_R\times \{t=-1\})$  and $(y,s)\in ((\R^n\setminus B_1)\times (-1,0])\cup (\R^n\times \{t=-1\})$,
\[
|u(x,t)-u(y,s)|\le \rho(|x-y|\vee |t-s|),
\]
then $|u-v|\le\va$ in $B_1\times (-1,0]$.
\end{lem}
\begin{proof}
 It follows from the proof of Theorem 5.6 in \cite{LD2} with modifications. But since the choice of norms are different, we include the proof for completeness. We argue by contradiction. Suppose the above lemma was false. Then there would be sequences $\sigma_k$, $I_0^{(k)}$, $I_1^{(k)}$, $I_2^{(k)}$, $\eta_k$, $u_k$, $v_k$ such that $\sigma_k\to \sigma\in [\sigma_0,2]$, $\eta_k\to 0$ and all the assumptions of the lemma are valid, but $\sup _{B_1\times (-1,0]}|u_k-v_k|\ge\va$.

Since $I_0^{(k)}$ is a sequence of uniformly elliptic translation invariant operators with respect to $\mathscr L(\sigma_k)$, by Theorem 5.5 in \cite{LD2} (and its proof) that we can take a subsequence, which is still denoted as $I_0^{(k)}$, that converges weakly to some nonlocal operator $I_0$, and $I_0$ is also translation invariant uniformly elliptic with respect to the class $\mathscr{L}_0(\sigma)$.

It follows from the boundary regularity Theorem 3.2 in \cite{LD2} that $u_k$ and $v_k$ have a modulus of continuity, uniform in $k$, in $\overline B_1\times [-1,0].$ Thus, $u_k$ and $v_k$ have a uniform (in $k$) modulus of continuity on $B_{R_k}\times [-1,0]$ with $R_k\to\infty$. We have subsequences of $\{u_k\}$ and $\{v_k\}$, which will be still denoted as $\{u_k\}$ and $\{v_k\}$, converge locally uniformly in $\R^n\times [-1,0]$ to $u$ and $v$, as well as in  $C(-1,0,L^1(\R^n,\omega))$ by dominated convergence theorem, respectively. Moreover, $u=v$ in $((\R^n\setminus B_1)\times[-1,0])\cup(B_1\times\{t=-1\})$, and $\sup_{B_1\times (-1,0)}|u-v|\ge\va$.

In the following, we are going to show in viscosity sense that
\be\label{eq:limit same}
 u_t-I_0(u,x,t)=0=v_t-I_0(v,x,t) \quad\mbox{in } B_1\times (-1,0].
\ee
Since $I_0$ translation invariant and $u=v$ in $((\R^n\setminus B_1)\times[-1,0])\cup(B_1\times\{t=-1\})$, we can conclude from Corollary 3.1 in \cite{LD2} that $u\equiv v$ in $B_1\times[-1,0]$, which is a contradiction.

The second equality of \eqref{eq:limit same} follows from Theorem 5.3 in \cite{LD2}.  To prove the first equality of \eqref{eq:limit same}, let $p$ be a second order parabolic polynomial touching $u$ from below at a point $(x,t)\in B_1\times (-1,0]$ in a neighborhood $B_r(x)\times (t-r,t]$. Since $u_k$ converges uniformly to $u$ in $\overline B_1\times [-1,0]$, for large $k$, we can find $(x_k,t_k)\in B_r(x)\times (t-r,t]$ and $d_k$ so that $p+d_k$ touch $u_k$ at $(x_k,t_k)$. Furthermore, $(x_k,t_k)\to (x,t)$ and $d_k\to 0$ as $k\to\infty$. Since $\pa_tu_k-I^{(k)}_2(u_k,x)\ge-\eta_k$, if we let
\[
  w_k(y,s)=\begin{cases}
  p+d_k\quad\mbox{in }B_r(x)\times (t-r,t];\\
   u_k\quad\mbox{in }(\R^n\setminus B_r(x))\times (B_r(x)\times \{s=t-r\}),
 \end{cases}
\]
we have $\pa_t  w_k(x_k,t_k)-I_2^{(k)}( w_k,x_k,t_k)\ge-\eta_k$, and
\[
 w=\lim_{k\to\infty}  w_k=\begin{cases}
  p\quad\mbox{in }B_r(x)\times (t-r,t];\\
   u\quad\mbox{in }(\R^n\setminus B_r(x))\times (B_r(x)\times \{s=t-r\}).
 \end{cases}
\]
Let $(z,s)\in B_{r/4}(x)\times (t-r/4, t]$. We have
\[
 \begin{split}
 & |I^{(k)}_2(  w_k,z,s)-I_0(  w,z,s)|\\
 &\le   |I^{(k)}_2( w_k,z,s)-I^{(k)}_2(  w,z,s)|+ |I^{(k)}_2(  w,z,s)-I_0( w,z,s)|\\
&\le \sup_{L\in\mathscr L(\sigma_k)} |L( w_k-  w)(z,s)|+|I^{(k)}_2(  w,z,s)-I_0(  w,z,s)|\\
&\le \int_{\R^n\setminus B_{r/2}}\frac{2\Lda|\delta (  w_k-  w)(z,y,s)|}{|y|^{n+\sigma_{k}}}\ud y+|I^{(k)}_2(  w,z,s)-I^{(k)}_0(  w,z,s)|\\
&\quad\quad+|I^{(k)}_0(  w,z,s)-I_0(  w,z,s)|.
 \end{split}
\]
Since $u_k$ are uniformly bounded in $\R^n\times[0,1]$, by dominated convergence theorem, the first term goes to $0$ as $k\to\infty$. Moreover, the convergence is uniform in $(z,s)$. Meanwhile, since $\|I^{(k)}_2-I^{(k)}_0\|_*\to 0$ in $B_1\times (-1,0]$ and $  w$ is bounded, we have that the second goes to $0$ uniformly for $(z,s)\in B_{r/4}(x)\times (t-r/4, t]$. Since  $I^{(k)}_0$ converges weakly to $I_0$, the third term also goes to zero uniformly for $(z,s)\in B_{r/4}(x)\times (t-r/4, t]$. Therefore, $I^{(k)}_2(  w_k,z,s)\to I_0(  w,z,s)$ uniformly in $(z,s)\in B_{r/4}(x)\times (t-r/4, t]$. Since $I_0w$ is continuous in $B_r(x)$, we can compute that
\[
\begin{split}
& |I^{(k)}_2(  w_k,x_k, t_k)-I_0(  w,x,t)|\\
 &\quad\le |I^{(k)}_2(  w_k,x_k,t_k)-I_0(  w,x_k,t_k)|+|I_0(  w,x_k,t_k)-I_0(  w,x,t)| \to 0
 \end{split}
\]
as $k\to\infty$. Since $\pa_t   w_k(x_k,t_k)-I_2^{(k)}(   w_k,x_k,t_k)\ge-\eta_k$ and $\pa_t   w_k(x_k,t_k)\to \pa_t  w(x,t)$, it follows that that $w_t(x,t)-I_0(w,x,t)\ge0$. Thus, $u_t(x,t)-I_0(u,x,t)\ge 0$ in viscosity sense. Similarly, we can show that $u_t(x,t)-I_0(u,x,t)\le 0$ in viscosity sense. This finishes the proof of the first equality of \eqref{eq:limit same}.
\end{proof}

\begin{lem}\label{lem:appr}
For some $\sigma\ge \sigma_0>0$ we consider nonlocal continuous operators $I_0$, $I_1$ and $I_2$ uniformly elliptic with respect to $\mathscr{L}_0(\sigma)$. Assume also that $I_0$ is translation invariant and $I_0 0=0$.

Given $M_2, M_3>0$, $\beta\in (0,1)$, and $\va>0$, there exists $\eta_2$ (small) so that if $u,v,I_0,I_1$ and $I_2$ satisfy
\[
\begin{split}
v_t-I_0 (v,x,t)&=0 \quad \mbox{in }B_1\times (-1,0],\\
u_t-I_1 (u,x,t)&\le \eta_2 \quad \mbox{in }B_1\times (-1,0],\\
u_t-I_2 (u,x,t)&\ge -\eta_2 \quad \mbox{in }B_1\times (-1,0],\\
\end{split}
\]
in viscosity sense, and
\[
\begin{split}
\|I_1-I_0\|_*&\le\eta_2\quad \mbox{in }B_1\times (-1,0],\\
\|I_2-I_0\|_*&\le\eta_2\quad \mbox{in }B_1\times (-1,0],\\
u&=v  \quad \mbox{in }((\R^n\setminus B_1)\times[-1,0])\cup(B_1\times\{t=-1\}),\\
|u|&\le M_2 \quad \mbox{in }\R^n\times[-1,0],\\
u&\equiv 0 \quad \mbox{in }(\R^n\setminus B_{2})\times [-1,0],\\
[u]_{C^\beta(B_{2-\tau}\times[-1,0])}&\le M_3\tau^{-4}\mbox{ for all } \tau\in (0,1),\\
\end{split}
\]
then $|u-v|\le\va$ in $B_1$.
\end{lem}
\begin{proof}
This lemma can be proved similarly to Lemma \ref{lem:appr0}. Suppose the above lemma was false. Then there would be sequences $\sigma_k$, $I_0^{(k)}$, $I_1^{(k)}$, $I_2^{(k)}$, $\eta_k$, $u_k$, $v_k$ such that $\sigma_k\to\sigma\in [\sigma_0,2]$, $\eta_k\to 0$ and all the assumptions of the lemma are valid, but $\sup _{B_1}|u_k-v_k|\ge\va$.

 Since $I_0^{(k)}$ is a sequence of uniformly elliptic operators, we can take a subsequence, which is still denoted as $I_0^{(k)}$, that converges weakly to some nonlocal operator $I_0$, and $I_0$ is also translation invariant and elliptic with respect to the class $\mathscr{L}_0(\sigma)$.

By our assumptions, it is clear that, up to a subsequence, $u_k$ converges locally uniformly in $B_2\times [-1,0]$. Since $u_k\equiv 0$ in $(\R^n\setminus B_2)\times [-1,0]$, it converges almost everywhere to some function $u$ in $\R^n\times [-1,0]$. By dominated convergence theorem, $u_k$ converges to $u$ in $C(-1,0;L^1(\R^n,\omega))$. Since $v_k$ is bounded and has a fixed modulus continuity on $((B_{3/2}\setminus B_1)\times [-1,0])\cup B_1\times \{t=-1\}$, then by Theorem 3.2 in \cite{LD2}, there is another modulus continuity that extends to $\overline B_1\times [-1,0]$. Hence, $v_k$ converges uniformly in $\overline B_{3/2}\times [-1,0]$, and thus, converges to some function $v$ almost everywhere in $\R^n\times [-1,0]$. Moreover, $u=v$ in $((\R^n\setminus B_1)\times[-1,0])\cup(B_1\times\{t=-1\})$, and $\sup_{B_1\times (0,1)}|u-v|\ge\va$.

It follows from the proof of \eqref{eq:limit same} that $u$ and $v$ solve the same equation $ u_t-I_0(u,x,t)=0=v_t-I_0(v,x,t)$ in  $B_1\times (-1,0]$. Then $u=v$, which is a contradiction.
\end{proof}

\bigskip

\bigskip

\noindent Tianling Jin

\noindent Department of Mathematics, The University of Chicago, 5734 S. University Ave, Chicago, IL 60637, USA\\[1mm]
Email: \textsf{tj@math.uchicago.edu}

\bigskip

\noindent Jingang Xiong

\noindent Beijing International Center for Mathematical Research, Peking University, Beijing 100871, China\\[1mm]
Email: \textsf{jxiong@math.pku.edu.cn}

\end{document}